\documentclass{article}
\usepackage[T2A]{fontenc}
\usepackage[utf8]{inputenc}
\usepackage{tikz,bussproofs,amssymb,amsmath,amsthm,MnSymbol,hyperref}
\usepackage{soul}
\usepackage{authblk}
\usepackage[backend=biber,style=alphabetic,sorting=nyt,doi=false,isbn=false,url=false,giveninits=true]{biblatex}

\EnableBpAbbreviations

\usetikzlibrary{tikzmark}

\newtheorem{theorem}{Theorem}
\newtheorem{lemma}{Lemma}[section]

\newtheorem{corollary}[lemma]{Corollary}
\newtheorem{proposition}[lemma]{Proposition}

\theoremstyle{remark}
\newtheorem{definition}{Definition}
\newtheorem{example}[lemma]{Example}
\newtheorem{remark}[lemma]{Remark}

\newcommand{\bl}{\begin{lemma}}
\newcommand{\el}{\end{lemma}}
\newcommand{\bt}{\begin{theorem}}
\newcommand{\et}{\end{theorem}}
\newcommand{\bcor}{\begin{corollary}}
\newcommand{\ecor}{\end{corollary}}
\newcommand{\bp}{\begin{proof}}
\newcommand{\ep}{\end{proof}}
\newcommand{\bpr}{\begin{proposition}}
\newcommand{\epr}{\end{proposition}}
\newcommand{\brem}{\begin{remark} \em}
\newcommand{\erem}{\end{remark}}
\newcommand{\bd}{\begin{definition} \em}
\newcommand{\ed}{\end{definition}}
\newcommand{\bex}{\begin{example} \em
}
\newcommand{\eex}{\end{example}}
\newcommand{\beq}{\begin{equation} }
\newcommand{\eeq}{\end{equation}}

\newcommand{\bi}{\begin{itemize}
  }
\newcommand{\ei}{\end{itemize}}
\newcommand{\ben}{\begin{enumerate} }
\newcommand{\een}{\end{enumerate} }

\newcommand{\refeq}[1]{(\ref{#1})}
\newenvironment{enumr}{

\begin{enumerate}     }{\end{enumerate}

}
\newcommand{\benr}{\begin{enumr}
  }
\newcommand{\eenr}{
\end{enumr}}

\newcommand{\ignore}[1]{}

\newcommand{\al}[1]{\forall #1\:}
\newcommand{\ex}[1]{\exists #1\:}

\newlength{\hilflh}

\newcommand{\naturals}{\mathbb{N}}

\renewcommand{\emptyset}{\varnothing}

\newcommand{\gw}{\omega}

\renewcommand{\phi}{\varphi}

\newcommand{\eqv}{\leftrightarrow}

\newcommand{\ol}{\overline}

\newcommand{\sir}[1]{\Sigma_{#1}\mbox{\sf -IR}}

\newcommand{\pir}[1]{\Pi_{#1}\mbox{\sf -IR}}

\newcommand{\PA}{\mathrm{PA}}
\newcommand{\EA}{\mathrm{EA}}

\renewcommand{\models}{\vDash}      %amsmodels

\renewcommand{\nmodels}{\nvDash}

\newcommand{\nat}{\naturals}
\newcommand{\Q}{\mathsf{Q}}

\renewcommand{\leq}{\leqslant}

\newcommand{\SQ}{\mathsf{SQ}}
\newcommand{\SR}{\mathsf{S}^{\Pi}}
\newcommand{\ST}{\mathsf{S}^{\Sigma}}

\author[1]{Lev D. Beklemishev}
\author[1]{Daniyar  S. Shamkanov}
\author[1,2,3]{Ivan N. Smirnov}
\affil[1]{Steklov Mathematical Institute of Russian Academy of Sciences, Moscow}
\affil[2]{Ivannikov Institute for System Programming of the Russian Academy of Sciences}
\affil[3]{Moscow Institute of Physics and Technology}

\title{Fragments of Arithmetic and Cyclic Proofs} 
\begin{document}
\maketitle

\section{Introduction}

Cyclic proof systems are well adapted for axiomatization of logics enhanced with fixed points, recursion or induction mechanisms. Such systems are known for extensions of linear logic with fixed points \cite{BaDoSa16, NoSaTa18}, for the modal $\mu$-calculus \cite{NiWa96, Sti14}, for the separation logic \cite{BrBoCa08, KiNaTeUn20}, as well as for the G\"{o}del-L\"{o}b provability logic \cite{Sham14} and arithmetic \cite{Sim17, BeTa17}. They seem to be promising for tasks of proof search, since they allow, in a sense, to search simultaneously for more complex kinds of inductions which may result in shorter proofs. So, the task of guessing the induction formula is replaced by the task of detecting possible cycles in a proof. James Brotherston et al.\ wrote a generic theorem prover \textsc{Cyclist} based on the cyclic proof format \cite{BrGoPe12}. Similar but weaker mechanisms are also known in the context of resolution based first order provers (such as \textsc{Vampire}), when these are enhanced by some basic forms of induction (see the analysis of some such systems in terms of the so-called clause set cycles by Hetzl and Vierling \cite{HeVi21}).  

Since the cyclic proof format is relatively new, a lot of research activity is currently underway also on the theoretical side of such proofs and non-well-founded proofs in general. Researchers are attracted both by the semantic aspect of  non-well-founded and cyclic systems \cite{San02, FoSa13, Sham20} and by studying them from the point of view of structural proof theory \cite{BaDoSa16, DaPo18, Sau23, OdBrTa23, SiStZe24, Sham24}. In addition, cyclic systems are applied to obtain interpolation properties \cite{Sham14, AfLeMe21}, as well as to prove realization theorems that relate modal and justification logics \cite{Sham16, Sham25}. A special and interesting direction is the study of cyclic proofs in the context of arithmetic \cite{Sim17, BeTa17, Das20} and the first-order logic with inductive definitions \cite{Bro05, BrSi11, BeTa19}.

The goal of this paper is to design an alternative cyclic proof system for Peano arithmetic that could be simpler than the existing ones and well-adapted both for proof analysis and for 
automatizing inductive proof search. In addition, we will show how various traditional subsystems of Peano arithmetic defined by restricted forms of induction can be represented as fragments of the proposed system.  

%\paragraph{Cyclic arithmetic.}
\bigskip  
A cyclic proof is a derivation tree where the hypotheses (leaves) are axioms, as in the usual proof, or are connected by `back-links' to the identical formulas or subsequents occurring \emph{below} in the proof. Tracing such a proof backwards from conclusions to premisses may result in infinite loops. Therefore, in order for such a proof figure to correspond to a correct argument (rather than a vicious circle), one imposes a global soundness condition in terms of the variables occurring in the proof. Details of such conditions can vary, but they can be roughly stated as requiring that every path in the infinite `unfolding' of a cyclic proof infinitely often goes through a rule that ensures the descent of the parameters (called \emph{the progress point}). This allows explicit induction rules in such a system to be abandoned in favor of the simpler \emph{case} rules (which correspond to the axiom that every natural number is either $0$ or a successor). 

The proofs using the induction axioms can be modeled by rather simple cyclic proofs where the back-links are independent from one another. However, the more complex cyclic proofs can be hard to transform into the usual induction proofs and establishing exact correspondence with the traditional Hilbert or Gentzen-style systems is often quite involved. 

There are three systems related to cyclic arithmetic in the current literature:
\begin{enumerate}
    \item The system CLKID of inductive definitions by Brotherston, later studied by Berardi and Tatsuta~\cite{BeTa17};
    \item The system CA by Simpson and its refinements C$\Sigma_n$ studied by Das \cite{Sim17,Das20}. These are cyclic versions of the first order Peano arithmetic $\PA$ and of its fragments;% defined by $\Pi_{n+1}$-consequences of $\Sigma_{n+1}$-induction schema;
    \item A cyclic version of G\"odel's system T from a recent paper by Das \cite{Das21}.
\end{enumerate}

In this paper we deal with a cyclic proof system for the first-order arithmetic. There are two main aspects in which the proposed system differs from those in the literature. Firstly, the soundness conditions in our system are simpler than those of CA and C$\Sigma_n$ which hopefully makes both the proof analysis and the proof search more straightforward. %They do not refer to all possible paths in the infinite unfolding of the proof. This condition happens to be definable in terms of B\"uchi automata, therefore constructive, yet it is less transparent than proof-checking in the more usual proof systems. 

Secondly, and more importantly, our proof system differs from the others in the variable discipline. As a result, by suitable restrictions of the quantifier complexity of formulas in the proof we obtain the familiar series of fragments of $\PA$ that are important in the metamathematics of Peano arithmetic, such as the fragments $\mathit{I\Sigma}_n$ defined by $\Sigma_n$-induction schemata and $\mathit{I\Sigma}_n^R$  axiomatized by $\Sigma_n$-induction rules. 

The systems C$\Sigma_n$ studied in \cite{Das20} are obtained by restricting the complexity of all formulas in a CA proof to the class $\Sigma_n$ or $\Pi_n$. However, under such a restriction the set of provable statements will not, in general, be closed under logical consequence, since the complexity of cuts is also restricted to $\Sigma_n$ or $\Pi_n$. To deal with this problem, Das defines C$\Sigma_n$ as the closure under the ordinary logical consequence of the set of formulas provable by $\Sigma_n$-restricted cyclic proofs. He then shows that the resulting theory corresponds to the fragment of Peano arithmetic defined by all $\Pi_{n+1}$-consequences of $\Sigma_{n+1}$-induction schema. In particular, the theory C$\Sigma_0$ can be axiomatized by the set of $\Pi_1$-consequences of $\mathit{I\Sigma}_1$, which is of consistency strength of $\mathit{I\Sigma}_1$ yet strictly weaker than the primitive recursive arithmetic  PRA and the equivalent fragment $\mathit{I\Sigma}_1^R$ in the language of $\PA$.  

An explicit Hilbert-style formulation of the theory axiomatized by $\Pi_{n+1}$-consequences of $\mathit{I\Sigma}_{n+1}$ is known. By the so-called \emph{Schmerl's formula} (see~\cite[Theorem 4]{Bek99b}) it is deductively equivalent to the uniform  $\Pi_{n+1}$-reflection schema over $\EA$ iterated $\gw^\gw$ times. 
This paper arose out of the natural question, inspired by the work of Das, which restrictions on cyclic proofs would yield systems deductively equivalent to the more fundamental fragments of Peano arithmetic such as $\mathit{I\Sigma_n}$ or $\mathit{I\Sigma_n^R}$. 

Our methods are considerably simpler than those of the previous papers. In particular, we use elementary syntactic arguments without any references to B\"uchi automata. This can be explained by the more explicit character of our notion of cyclic proof. 

The plan of the paper is as follows. In Section \ref{prelim}, we introduce necessary preliminaries
on first order arithmetic and specify a Tait-style sequent calculus for Robinson's arithmetic $\Q$. In Section \ref{infinit}, we introduce non-well-founded proof systems and show that they are strong enough to prove all true arithmetical sentences. In Section~\ref{cyclic}, we introduce variable discipline using annotated proofs and define cyclic proofs. Finally, in Sections~\ref{uncycle} and \ref{rules}, we relate the cyclic systems to the main fragments of Peano arithmetic such as $\mathit{I\Sigma_n}$, $\mathit{I\Pi}^R_{n}$ and $\mathit{I\Sigma}_n^R$. 

\section{Preliminaries} \label{prelim}

We assume familiarity with the standard Hilbert-style axiomatization of first order Peano arithmetic following, e.g., \cite{HP}. The language of Peano arithmetic has constant $0$, function symbols $s$ (successor), $+$ and $\cdot$, and relation symbol $=$. The formula $x\leq y$ is an abbreviation for $\ex{z}(z+x=y)$.  

\emph{Bounded quantifiers} are abbreviations for
 \begin{gather*}
     \forall y \leq t \; \varphi \coloneq \forall y \; (y \leq t \rightarrow \varphi), \qquad \exists y \leq t \; \varphi \coloneq \exists y \; (y \leq t \wedge \varphi),
 \end{gather*}
 where $t$ is a term and $y$ does not occur in $t$. A \emph{bounded formula} is an arithmetical formula built from atomic formulas by logical connectives and bounded quantifiers only. The set of all such formulas is denoted $\Delta_0$. 
 
 The classes of $\Sigma_n$ and $\Pi_n$ formulas are inductively defined as follows: We let $\Sigma_0 = \Pi_0=\Delta_0$.
 An arithmetical formula is $\Sigma_{n+1}$ ($\Pi_{n+1}$) if it is obtained from $\Pi_n$ ($\Sigma_n$) formulas using existential (universal) quantifiers and positive propositional connectives ($\land$ and $\lor$).

\emph{Robinson's arithmetic} $\Q$ is axiomatized by the axioms and rules of the first-order predicate logic with equality as well as by (the universal closures of) the following basic axioms:

\ben 
\item $\neg s(x)=0$; \quad $s(x)=s(y)\to x=y$;
\item $x=0\lor \ex{y} x= s(y)$;
\item $x+0=x$; \quad $x+s(y)=s(x+y)$;
\item $x\cdot 0=0$; \quad $x\cdot s(y)=x\cdot y+ x$.
\een 

\emph{Peano arithmetic} $\PA$ is axiomatized by $\Q$ together with the induction schema 
$$\phi(0)\land \al{x}(\phi(x)\to \phi(s(x)) \to \al{x}\phi(x), \leqno \mathrm{(Ind)}$$
for all formulas $\phi$ (possibly containing free variables different from $x$). 

The theories $\mathit{I\Sigma}_n$ ($\mathit{I\Pi}_n$) are obtained by extending $\Q$ by the induction schema $\mathrm{(Ind)}$ for all $\Sigma_n$ (respectively, $\Pi_n$) formulas $\phi$. It is well-known that $\mathit{I\Sigma}_n$ and $\mathit{I\Pi}_n$ are deductively equivalent. Theories $\mathit{I\Sigma}_n$ constitute a strictly increasing hierarchy of fragments of $\PA$; the union of this hierarchy is $\PA$ itself. For all $n>0$, theories $\mathit{I\Sigma}_n$ are finitely axiomatizable.\footnote{For $n=0$ this is a well-known open problem.}  

The \emph{induction rule} is the inference rule 
$$\frac{\phi(0), \quad \al{x}(\phi(x)\to \phi(s(x))}{\al{x}\phi(x)}. \leqno (\mathrm{IR})$$

It is well-known that the closure of $\Q$ under $(\mathrm{IR})$ is equivalent to Peano arithmetic. We let $\mathit{I\Sigma}_n^R$ ($\mathit{I\Pi}_n^R$) denote the theory axiomatized over $\Q$ by $(\mathrm{IR})$ restricted to $\Sigma_n$ ($\Pi_n$) formulas $\phi$. It is known that $\mathit{I\Sigma}_n^R$ is deductively equivalent to $\mathit{I\Pi}_{n+1}^R$ and strictly weaker than $\mathit{I\Sigma}_n$, for $n>0$. For $n=0$, the fragment  $I\Delta_0^R$ is equivalent to $I\Delta_0$. 

None of the theories $\mathit{I\Sigma}_n^R$ for $n>0$ is finitely axiomatizable, and it is known that $\mathit{I\Sigma}_n$ conservatively extends $\mathit{I\Sigma}_n^R$ for $\Pi_{n+1}$ sentences. See~\cite{Bek97a} for more information on the fragments of arithmetic axiomatized by induction rules.    

%\section{A Tait-style sequent calculus}
\bigskip 
Next we introduce a proof system $\SQ$ for Robinson's arithmetic $\Q$ based on a one-sided variant of a sequent calculus for predicate logic due to Tait (see~\cite{Sch77}). We partly follow Negri and von Plato's approach to sequent calculi for axiomatic theories (see \cite{NePl98} or Chapter 6 of  \cite{NePlRa01}) to incorporate the rules for equality axioms and the mathematical axioms of $\Q$.

Recall that \emph{arithmetical terms} are built from a countable set of individual variables $\mathit{Var}=\{x_0, x_1, x_2 \dotsc \}$ and the constant $0$ by means of the unary function symbol $s$ and the binary symbols $+$ and $\cdot$ in a standard way. \emph{Arithmetical formulas}, denoted by $\varphi$, $\psi$, etc., are built up as follows:
\[ \varphi ::= t=t \mid t\neq t \mid (\varphi \wedge \varphi) \mid (\varphi \vee \varphi) \mid \forall y\; \varphi \mid \exists y\; \varphi \;, \]
where $t$ stands for arithmetical terms and $y$ represents an arbitrary individual variable.

The \textit{negation} $\overline{\varphi}$ of a formula $\varphi$ in Tait calculus is a syntactic operation defined by the law of double negation, de Morgan's laws and the duality laws for quantifiers, i.e., we inductively define:
\begin{gather*}
\overline{t_0=t_1} \coloneq t_0\neq t_1,\qquad
\overline{t_0\neq t_1} \coloneq t_0=t_1,\\
\overline{(\varphi_0 \wedge \varphi_1)} \coloneq (\overline{\varphi_0} \vee \overline{\varphi_1}),\qquad
\overline{(\varphi_0 \vee \varphi_1)} \coloneq (  \overline{\varphi_0} \wedge \overline{\varphi_1} ) ,\\
\overline{\forall y\; \varphi} \coloneq \exists y\; \overline{\varphi},\qquad
\overline{\exists y\; \varphi} \coloneq \forall y\; \overline{\varphi}.
\end{gather*}
Note that the operation of negation is idempotent. In other words, $\overline{\overline{\varphi}}$ is equal to $\varphi$ for any formula $\varphi$.
Also, we put
\begin{gather*}
\top \coloneq 0=0, \quad \bot\coloneq 0\neq 0, \quad \varphi\rightarrow \psi \coloneq \overline{\varphi} \vee \psi, \quad \varphi \leftrightarrow \psi \coloneq (\varphi\rightarrow \psi) \wedge
(\psi\rightarrow \varphi).
\end{gather*}

We define \textit{sequents} as finite multisets of arithmetical formulas. For a sequent $\Gamma $, its intended interpretation as a formula is $\bigvee \Gamma$, where $\bigvee \emptyset \coloneq \bot$. Recall that sequents are often written without any curly braces, and the comma in the expression $\Gamma, \Delta$ means the multiset union.

The initial sequents and inference rules of $\SQ$ are:
\begin{gather*} \mathsf{(ax_a)}\quad 
\Gamma, t_0=t_1, t_0\neq t_1 , \qquad \mathsf{(ax_s)}\quad \Gamma, s(t)\neq 0 ,
\end{gather*}
\begin{gather*}
\AXC{$\Gamma , \varphi$}
\AXC{$\Gamma , \psi$}
\LeftLabel{$\mathsf{(\wedge)}$}
\RightLabel{ ,}
\BIC{$\Gamma , \varphi \wedge \psi$}
\DisplayProof \qquad
\AXC{$\Gamma , \varphi ,\psi $}
\LeftLabel{$\mathsf{(\vee)}$}
\RightLabel{ ,}
\UIC{$\Gamma , \varphi \vee \psi $}
\DisplayProof
\end{gather*}
\begin{gather*}
\AXC{$\Gamma , \varphi [y\mapsto z]  $}
\LeftLabel{$\mathsf{(\forall)}$}
\RightLabel{($z\nin \textit{FV}(\Gamma,\forall y\; \varphi)$),}
\UIC{$\Gamma , \forall y\; \varphi $}
\DisplayProof \qquad
\AXC{$\Gamma  , \varphi [y\mapsto t] , \exists y\; \varphi $}
\LeftLabel{$\mathsf{(\exists)}$}
\RightLabel{ ,}
\UIC{$\Gamma , \exists y\; \varphi  $}
\DisplayProof
\end{gather*}
\begin{gather*}
\AXC{$\Gamma  , t \neq t $}
\LeftLabel{$\mathsf{(ref)}$}
\RightLabel{ ,}
\UIC{$\Gamma  $}
\DisplayProof \qquad
\AXC{$\Gamma , t_0\neq t_1, (u_0\neq u_1)[y\mapsto t_0]  , (u_0\neq u_1)[y\mapsto t_1]  $}
\LeftLabel{$\mathsf{(rep)}$}
\RightLabel{ ,}
\UIC{$\Gamma , t_0\neq t_1, (u_0\neq u_1)[y\mapsto t_0]  $}
\DisplayProof 
\end{gather*}
\begin{gather*}
\AXC{$\Gamma , t+0\neq t $}
\LeftLabel{$\mathsf{(add}_0)$}
\RightLabel{ ,}
\UIC{$\Gamma  $}
\DisplayProof \qquad
\AXC{$\Gamma , t+s(u)\neq s(t+u) $}
\LeftLabel{$(\mathsf{add}_s)$}
\RightLabel{ ,}
\UIC{$\Gamma  $}
\DisplayProof 
\end{gather*}
\begin{gather*}
\AXC{$\Gamma , t\cdot 0 \neq 0 $}
\LeftLabel{$(\mathsf{mult}_0)$}
\RightLabel{ ,}
\UIC{$\Gamma  $}
\DisplayProof \qquad
\AXC{$\Gamma , t\cdot s(u)\neq t\cdot u +t $}
\LeftLabel{$(\mathsf{mult}_s)$}
\RightLabel{ ,}
\UIC{$\Gamma  $}
\DisplayProof 
\end{gather*}
\begin{gather*}
\AXC{$\Gamma , s(t_0)\neq s(t_1), t_0\neq t_1 $}
\LeftLabel{$(\mathsf{pred})$}
\RightLabel{ ,}
\UIC{$\Gamma , s(t_0)\neq s(t_1) $}
\DisplayProof \qquad
\AXC{$\Gamma [y\mapsto 0]$}
\AXC{$\Gamma [y\mapsto s(y)]$}
\LeftLabel{$\mathsf{(case)}$}
%\RightLabel{($z\nin \mathit{FV}(\Gamma [y\mapsto t] )$ or $z=t$),}
\RightLabel{($y\in \textit{FV}(\Gamma)$),}
\BIC{$\Gamma $}
%\BIC{$\Gamma $}
\DisplayProof
\end{gather*}
\begin{gather*}
\AXC{$\Gamma   $}
\LeftLabel{$(\mathsf{weak})$}
\RightLabel{ ,}
\UIC{$\Gamma  , \Delta$}
\DisplayProof \qquad
\AXC{$\Gamma , \varphi  $}
\AXC{$\Gamma , \overline{\varphi}  $}
\LeftLabel{$(\mathsf{cut})$}
\RightLabel{ .}
\BIC{$\Gamma $}
\DisplayProof
\end{gather*}
Here,  $\psi [y\mapsto t]$ ($\Gamma [y\mapsto t]$) denotes the formula (the sequent) obtained from $\psi$ ($\Gamma$) by replacing all free occurrences of $y$ in $\psi$ (in formulas of $\Gamma$) with the term $t$. As usual, we require that no variable of $t$ becomes bound after the replacement. In what follows, if the variable $y$ is clear from the context, we will often write $\psi [y\mapsto t]$ and $\Gamma [y\mapsto t]$ as $\psi (t)$ and $\Gamma (t)$, respectively. 

For the inference rules $(\wedge)$, $(\vee)$, $(\forall)$ and $(\exists)$, the formula explicitly displayed in the conclusion of a rule is called the \emph{principal formula} of the corresponding inference. For the rules $(\mathsf{\forall})$ (or $\mathsf{(case)}$), $z$ (or $y$) is called the \emph{active variable} of the given inference. 

The rules $(\land), (\lor), (\exists), (\forall), (\mathsf{weak})$ and $(\mathsf{cut})$ are the standard rules of the Tait calculus. The rules $(\mathsf{ref})$ and $(\mathsf{rep})$ account for the equality axioms. The rest of the rules and $(\mathsf{ax}_s)$ correspond to the mathematical axioms of $\Q$.  

We also emphasize, though it is inessential for what follows, that we have chosen this slightly uncommon formulation of the arithmetical rules of $\SQ$ since the proofs constructed according to these rules, excluding the rule $\mathsf{(case)}$, enjoy cut elimination \cite{NePl98, NePlRa01}.  

The following lemma is standard, so we essentially leave it without proof. 

\begin{lemma}\label{phiphi}
For any formula $\varphi$ and any finite multiset of formulas $\Gamma$, 
\ben
\item $\SQ\vdash \Gamma, \varphi, \overline{\varphi}$;
\item $\SQ\vdash \Gamma, \phi$ whenever $\SQ\vdash \Gamma,\phi,\phi$.
\een 
\end{lemma}

Using Lemma~\ref{phiphi} it is routine to prove that the following proposition holds.

\bpr \label{Q} \text{ }
\ben
\item If $\Q\vdash \phi$, then $\SQ\vdash \phi$;
\item If $\SQ\vdash \Gamma$, then $\Q\vdash \bigvee\Gamma$.
\een
\epr 

\bp We only notice that the rule (\textsf{case}) is equivalent to Axiom 2 of $\Q$ modulo the other rules. 
\ep 

\section{Non-well-founded proofs} \label{infinit}
In this section, we introduce a series of arithmetical sequent calculi $\mathsf{S}_n$, which allow non-well-founded proofs. These systems will have the axioms and inference rules of $\SQ$, but differ only in the notion of proof. 

\begin{definition} An \emph{(unrestricted) $\infty$-proof} is a possibly infinite tree whose nodes are marked by sequents and that is constructed according to the rules of the calculus $\SQ$. In addition, every leaf in an $\infty$-proof is marked by an initial sequent of $\SQ$ and, for every infinite branch in it, there is a tail of the branch satisfying the following conditions: 
(a) the tail does not pass through left premisses of the rule $\mathsf{(case)}$, (b) if the tail passes through the right premiss $\Gamma [z\mapsto s(z)]$ of the rule $\mathsf{(case)}$, then, after that application, the tail does not intersect applications of the rule ($\forall$) with the active variable $z$, (c) for some variable $y$, the tail passes through the right premiss of the rule $\mathsf{(case)}$ with the active variable $y$ infinitely many times. 
\end{definition}

\begin{definition} An $\infty$-proof is \emph{$\Pi_n$-restricted} if every infinite branch in it has a tail satisfying all the conditions above and the following ones: 
(d) if the tail passes through the left (right) premiss of the rule ($\wedge$), then the formula $\varphi$ ($\psi$) belongs to $\Pi_{n}$, (e) if 
the tail intersects an application of the rule ($\forall$), then the formula $\varphi [y\mapsto z]$ belongs to $\Pi_{n}$, (f)
if the tail passes through the left (right) premiss of the rule ($\mathsf{cut}$), then the cut formula $\varphi$ ($\overline{\varphi}$) is from $\Pi_{n}$,
(g) if the tail passes through the right premiss $\Gamma [z\mapsto s(z)]$ of the rule $\mathsf{(case)}$, then the variable $z$ can occur freely only in $\Pi_{n}$-formulas of $\Gamma$.
\end{definition}

As expected, the following soundness lemma holds for unrestricted $\infty$-proofs. 

\bl If $\Gamma$ has an $\infty$-proof, then $\nat\models v(\bigvee \Gamma)$ for all assignments $v\colon \mathit{Var}\to \mathbb{N}$. 
\el 
\bp\ Assume, on the contrary, $v(\bigvee\Gamma)$ is false for some $v$. In this case $\Gamma$ cannot be an axiom of $\SQ$. All the inference rules of $\SQ$ preserve the validity of sequents in $\nat$. Hence, there exists an infinite branch $P$ in the proof tree such that all sequents in $P$ are not valid.  

We construct $P$ inductively by working from the conclusion to the premisses of the rules. Denote by $\Gamma_i$ the $i$-th sequent in $P$ counting from $\Gamma_0:=\Gamma$. Let $v_i$ be the assignment such that $\nat\nmodels v_i(\bigvee\Gamma_i)$. We can assume that $v_{i+1}$ is different from $v_i$ only in the following cases. If $\Gamma_i$ is obtained from $\Gamma_{i+1}$ by a $(\forall)$ rule application, then one must redefine the assignment of the active variable so that $\Gamma_{i+1}$ is false. If $\Gamma_{i}$ is obtained by $\mathsf{(case)}$ rule and $\Gamma_{i+1}$ is its left (right) premiss, then one must put $v_{i+1}(y)=0$ ($v_{i+1}(y)=v_i(y)-1$) for the active variable $y$. In all other cases, we may assume $v_{i+1}(x)=v_i(x)$. 

Consider the tail $T$ of the branch $P$ satisfying conditions (a)--(c), and let $y$ be the variable active in infinitely many applications of $\mathsf{(case)}$ in $T$ according to (c). Taking a shorter tail if necessary we can assume that $T$ begins with an application of $\mathsf{(case)}$ with $y$ active, hence by condition (b) $y$ will not be active in applications of $(\forall)$ anywhere in $T$. It follows that the assignment $v_i(y)$ will only be changed by applications of $\mathsf{(case)}$ rule and is weakly decreasing with $i$. But $T$ goes through infinitely many such applications, a contradiction. \ep 

Our next observation shows that the set of sequents having a $\Pi_n$-restricted $\infty$-proof is closed under a version of $\Pi_n$-restricted $\gw$-rule with side formulas. We denote by $\bar k$ the numeral for the number $k\in\nat$, that is, the term $s(s(\dots(0)\dots))$ ($k$ symbols $s$). 

\bl \label{omrule} Suppose, for each $k$, $\pi_k$ is an $\infty$-proof of $\Gamma,\phi(\bar k)$. Then there is an $\infty$-proof $\pi$ of $\Gamma,\phi(x)$ where $x\notin \mathit{FV}(\Gamma)$. Moreover, if $\phi$ is $\Pi_n$ and all $\pi_k$ are $\Pi_n$-restricted then so is $\pi$. 
\el 

\begin{proof}
    Let $x\notin\mathit{FV}(\Gamma)$. If $x$ does not actually occur free in $\phi(x)$ the claim is trivial. Otherwise, the proof $\pi$  looks as follows:
\begin{equation*}
     \label{Example1}
\AXC{$\pi_0$}
\noLine
\UIC{$\vdots$}
\noLine
\UIC{$\Gamma, \varphi (0)$}
\AXC{$\pi_1$}
\noLine
\UIC{$\vdots$}
\noLine
\UIC{$\Gamma,\varphi (s(0))$}
\AXC{$\pi_2$}
\noLine
\UIC{$\vdots$}
\noLine
\UIC{$\Gamma,\varphi(s(s(0))$}
\AXC{$\dots$}
\LeftLabel{$\mathsf{(case)}$}
\BIC{$\Gamma, \varphi (s(s(x))$}
\LeftLabel{$\mathsf{(case)}$}
\BIC{$\Gamma , \varphi (s(x))$}
\LeftLabel{$\mathsf{(case)}$}
\RightLabel{ .}
\BIC{$\Gamma, \varphi (x)$}
\DisplayProof 
\end{equation*}
Observe that there is one new infinite branch in $\pi$, the rightmost one, and it clearly satisfies the necessary conditions.  
\end{proof}

From this lemma we obtain that $\infty$-proofs are sufficient to prove all true arithmetical sentences, even under the strongest assumption of being $\Pi_0$-re\-stricted. We call a sequent $\Gamma$ \emph{true} if $\Gamma$ is a set of  sentences such that $\nat\models\bigvee\Gamma$.

\bl  Every true sequent $\Gamma$ has a $\Pi_0$-restricted $\infty$-proof. 
\el 

\bp\ Within this proof we will say that a sequent is \emph{provable} if it has a  $\Pi_0$-restricted $\infty$-proof. By induction on $n$, we show that every true sequent $\Gamma\subseteq \Pi_{2n}$ is provable. By weakening, it is sufficient to prove that every true $\Pi_{2n}$-sentence is provable. 
For the case $n=0$, we just recall that $\Q$ proves every true bounded sentence and apply Proposition~\ref{Q}.

Assume the claim holds for $n$, which we call the \emph{main induction hypothesis}. We show that every true $\Pi_{2n+2}$ sentence is provable by a \emph{secondary induction} on the build-up of it. 

First, as the secondary induction basis, we prove that every true $\Sigma_{2n+1}$ sentence is provable. We argue by a subinduction on the build-up of such a sentence. For $\Pi_{2n}$ sentences this holds by the main induction hypothesis. For sentences obtained by conjunction, disjunction or an existential quantifier, the claim easily follows using the inference rules $(\land), (\lor)$, $(\exists)$ and ($\mathsf{weak}$). 

Second, as the secondary induction step, if $\Pi_{2n+2}$ sentences $\phi$ and (or) $\psi$ are provable, then so are $\phi\land\psi$ (respectively, $\phi\lor\psi$), by the rules $\mathsf{(\land)}$, $\mathsf{(\lor)}$ and ($\mathsf{weak}$). 

Finally, we consider a true $\Pi_{2n+2}$ sentence of the form $\al{x}\psi(x)$ with  $\psi$ in $\Pi_{2n+2}$. We must show that $\al{x}\psi(x)$ is provable.

First, we remark that $\SQ$ proves the sequent 
$$\ol{\psi(\bar m)}, x\neq \bar m, \psi(x),$$
which is an instance of the equality schema valid in predicate logic. By the secondary induction hypothesis, for each $m$, the formula $\psi(\bar m)$ is provable. Hence, by $\mathsf{(cut)}$, for each $m$ we have a proof of $x\neq \bar m, \psi(x)$. Applying Lemma \ref{omrule} we obtain a proof of $x\neq z,\psi(x)$, where $z$ is a fresh variable. (The proof is $\Pi_0$-restricted since $x\neq z$ is $\Pi_0$.) It follows that $\al{z} x\neq z,\psi(x)$ is provable, whence by a $\mathsf{(cut)}$ with the formula $\ex{z} x=z$ provable in predicate logic we obtain $\psi(x)$, so $\al{x} \psi (x)$ follows by $\mathsf{(\forall)}$. 
\ep

%\ignore

%Since $\Q$ proves $\bar k\leq x \lor x\leq \bar k$, $\SQ$ proves the sequent $$\bar k\leq x, x\leq \bar k.$$ Then by cut the sequent
  %$$\psi(x) , \bar k\leq x$$
%is provable, for all $k$.   
%By Lemma~\ref{omrule} we obtain that there is a proof of $$\psi(x) , \al{z}{z\leq x}.$$ Since $\Q$ proves the negation of $\al{z}{z\leq x}$ we obtain the desired conclusion $\psi(x)$ by cut. 
%\ep 

We conclude that the notion of $\Pi_0$-restricted $\infty$-proof is as powerful as the (unrestricted) $\omega$-proof in first order arithmetic. 

Now we introduce a series of arithmetical sequent calculi $\mathsf{S}_n$. In what follows, we also call $\Pi_{n+1}$-restricted $\infty$-proofs \emph{$\infty$-proofs of $\mathsf{S}_n$}. However, in order to obtain notions of proof adequate for axiomatizable theories, such as Peano arithmetic, we shall restrict the set of $\infty$-proofs to \emph{regular} ones. An $\infty$-proof is called \emph{regular} if it contains only finitely many non-isomorphic subtrees with respect to the marking of sequents.
We say that a sequent $\Gamma$ is \emph{provable in $\mathsf{S}_n$} if there is a regular $\infty$-proof of $\mathsf{S}_n$ with the root marked by $\Gamma $. 

Let us consider an example of an infinite regular $\infty$-proof of $\mathsf{S}_n$
{\small\begin{gather*}\label{Example3}
\AXC{$\sigma_0$}
\noLine
\UIC{$\vdots$}
\noLine
\UIC{$\overline{\varphi(0)}, \psi , \varphi (0)$}
\AXC{$\pi$}
\noLine
\UIC{$\vdots$}
\noLine
\UIC{$\overline{\varphi(0)}, \psi ,\varphi (x)$}
\LeftLabel{$\mathsf{(weak)}$}
\UIC{$\overline{\varphi(0)}, \psi , \varphi (x) , \varphi (s(x))$}
\AXC{$\sigma_1$}
\noLine
\UIC{$\vdots$}
\noLine
\UIC{$\overline{\varphi(0)}, \psi ,\overline{\varphi (s(x)) }, \varphi (s(x))$}
\LeftLabel{$\mathsf{(\wedge)}$}
\BIC{$\overline{\varphi(0)}, \psi ,  \varphi (x) \wedge \overline{\varphi (s(x)) }, \varphi (s(x))$}
\LeftLabel{$\mathsf{(\exists)}$}
\UIC{$\overline{\varphi(0)}, \psi , \varphi (s(x))$}
\LeftLabel{$\mathsf{(case)}$}
\RightLabel{ ,}
\BIC{$\overline{\varphi(0)},\psi , \varphi (x)$}
\DisplayProof 
\end{gather*}}
where the subtree $\pi$ is isomorphic to the whole $\infty$-proof, $\sigma_0$ and $\sigma_1$ are given by Lemma \ref{phiphi}, $\psi$ is $  \overline{\forall x \; (\varphi (x) \to \varphi (s(x)) )}$ and $\varphi$ is an arbitrary $\Pi_{n+1}$-formula such that $x\in \mathit{FV}(\varphi)$. Here the unique infinite branch passes through the right premiss of the rule $\mathsf{(case)}$ with the active variable $x$ infinitely many times. We immediately see that the required conditions on infinite branches hold if we consider the given branch as its own tail.  

The example above essentially shows that any instance of $\Pi_{n+1}$-induction schema is provable in $\mathsf{S}_n$, so we immediately obtain 

\begin{corollary} \label{sn}
Every theorem of $\mathit{I\Sigma}_{n+1}$ is provable in $\mathsf{S}_n$.     
\end{corollary}

\section{Cyclic proofs and annotations} \label{cyclic}
%Now we present annotated versions of the systems $\mathsf{S}_n$ and define useful finite representations of regular $\infty$-proofs called cyclic (or circular) proofs.

%establish that $\mathsf{S}_n$ is embeddable into the well-known system $\mathit{I\Sigma}_n$ for every $n\in\mathbb{N}$. In order to facilitate our prove of the realization theorem, we introduce annotated versions of sequents and inference rules of the sequent calculus $\mathsf{S}$. We also .

In this section, in order to facilitate our treatment of provability in $\mathsf{S}_n$, we introduce annotated
versions of sequents and inference rules of this calculus. We also define finite representations of regular $\infty$-proofs called cyclic (or circular) proofs.

An \emph{annotated sequent} is an expression of the form {$\Gamma \upharpoonright V$}, where $\Gamma$ is an ordinary sequent and $V$ is a finite set of individual variables. %In addition, if $\alpha$ is a formula, then the musltiset $\Delta$ must contain $\alpha$. We also require that negative occurrences of modal connectives in $\Gamma \Rightarrow_\alpha \Delta$ (i.e. in $\bigwedge \Gamma \rightarrow \bigvee \Delta$) are labelled with even natural numbers and positive ones are labelled with odd numbers. An annotated sequent $\Gamma \Rightarrow_\alpha \Delta$ is called \emph{properly annotated} if the formula $\bigwedge \Gamma \rightarrow \bigvee \Delta$ is properly annotated. Here is an example of a properly annotated sequent:
Annotated versions of initial sequents of $\mathsf{S}_n$ are $\Gamma, t_0=t_1, t_0\neq t_1 \upharpoonright V $ and $\Gamma, s(t)\neq 0 \upharpoonright V$.
The annotated versions of inference rules of $\mathsf{S}_n$, except the rules ($\wedge$), ($\forall$), $\mathsf{(case)}$ and ($\mathsf{cut}$), are obtained by annotating the premiss and the conclusion of a rule so that they have the same annotation $V$. 
%\begin{gather*}
%\Gamma, t_0=t_1, t_0\neq t_1 \upharpoonright V , \qquad \Gamma, s(t)\neq 0 \upharpoonright V,
%\end{gather*}
We define annotated versions of ($\wedge$), ($\forall$) and $\mathsf{(case)}$ as
\begin{gather*}
\AXC{$\Gamma , \varphi \upharpoonright V_0$}
\AXC{$\Gamma , \psi \upharpoonright V_1$}
\LeftLabel{$\mathsf{(\wedge)}$}
\RightLabel{ ,}
\BIC{$\Gamma , \varphi \wedge \psi \upharpoonright V$}
\DisplayProof \qquad
\AXC{$ \Gamma , \varphi [y\mapsto z] \upharpoonright V_2$}
\LeftLabel{$\mathsf{(\forall)}$}
\RightLabel{($z\nin {\mathit{FV}(\Gamma,\forall y\; \varphi)}$),}
\UIC{$ \Gamma , \forall y\; \varphi \upharpoonright V$}
\DisplayProof \\\\
\AXC{$\Gamma [y\mapsto 0]\upharpoonright \emptyset$}
\AXC{$\Gamma [y\mapsto s(y)]\upharpoonright  V_3$}
\LeftLabel{$\mathsf{(case)}$}
\RightLabel{($y\in \mathit{FV}(\Gamma)$),}
\BIC{$\Gamma \upharpoonright V$}
%\BIC{$\Gamma $}
\DisplayProof
\end{gather*}
where
\begin{gather*}
V_0 \coloneq
\begin{cases}
V & \text{if $\varphi \in \Pi_{n+1}$,}\\
\emptyset & \text{otherwise,}
\end{cases} \qquad
V_1 \coloneq
\begin{cases}
V & \text{if $ \psi\in \Pi_{n+1}$,}\\
\emptyset & \text{otherwise,}
\end{cases} \\\\
V_2 \coloneq
\begin{cases}
V & \text{if $z\nin V$ and $\varphi [y\mapsto z]\in \Pi_{n+1}$,}\\
\emptyset & \text{otherwise,}
\end{cases}\\\\
V_3 \coloneq
\begin{cases}
V \cup \{y\} & \text{if $y$ freely occurs only in $\Pi_{n+1}$-formulas of $\Gamma$,}\\
\emptyset & \text{otherwise.}
\end{cases}
\end{gather*}
The annotated version of ($\mathsf{cut}$) has the form
\begin{gather*}
\AXC{$\Gamma , \varphi \upharpoonright V_0  $}
\AXC{$\Gamma , \overline{\varphi} \upharpoonright V_1  $}
\LeftLabel{$\mathsf{(cut)}$}
\RightLabel{ ,}
\BIC{$\Gamma \upharpoonright V$}
\DisplayProof
\end{gather*}
where
\begin{gather*}
V_0 \coloneq
\begin{cases}
V & \text{if $\varphi \in \Pi_{n+1}$,}\\
\emptyset & \text{otherwise,}
\end{cases} \qquad
V_1 \coloneq
\begin{cases}
V & \text{if $\overline{\varphi} \in \Pi_{n+1}$,}\\
\emptyset & \text{otherwise.}
\end{cases}
\end{gather*}

An \emph{annotated $\infty$-proof of $\mathsf{S}_n$} is a (possibly infinite) tree whose nodes are marked
by annotated sequents and that is constructed according to the annotated versions of inference rules of $\mathsf{S}_n$. In addition, all leaves of an annotated $\infty$-proof are marked by annotated initial sequents, and every infinite branch in it must contain a tail satisfying the following conditions:
(i) there are no sequents in the tail annotated with $\emptyset$ and
(ii) there is a variable $y$ such that the tail passes through the right premiss of the rule $\mathsf{(case)}$ with the active variable $y$ infinitely many times. An annotated $\infty$-proof is \emph{regular} if it contains only finitely many non-isomorphic subtrees with respect to annotations.

Note that if we erase all annotations in an annotated $\infty$-proof of the system $\mathsf{S}_n$, then the resulting tree is a normal $\infty$-proof of the same system, i.e. a $\Pi_{n+1}$-restricted $\infty$-proof. Let us prove the converse.

\begin{lemma}\label{AnnLem}
{Suppose $\pi$ is an $\infty$-proof of $\mathsf{S}_n$ with the root marked by $\Gamma$ and $V$ is a finite set of individual variables. Then  $\pi$ can be annotated so that the root of the resulting tree is marked by $\Gamma \upharpoonright V$. Moreover, the obtained annotated $\infty$-proof is finite (regular) if $\pi $ is finite (regular).}
\end{lemma}
\begin{proof}
Note that, for any application of an inference rule of $\mathsf{S}_n$ and any annotation of its conclusion, one can annotate its premisses and obtain an application of the annotated version of the rule. Moreover, the choice of annotations for the premisses is unique. 

Now assume we have an  $\infty$-proof $\pi$ of $\mathsf{S}_n$ with the root marked by a sequent $\Gamma$. We annotate $\Gamma$ in $\pi$ with the set $V$ and, moving upwards away from the root, replace all applications of inference rules in $\pi$ with their annotated versions. Let us check that the given tree, denoted by $\pi^\prime$, is an annotated $\infty$-proof. Consider an arbitrary infinite branch of $\pi^\prime$. We must show that there exists a tail $T'$ of the branch satisfying the conditions: 
\ben 
\item[(i)] there are no sequents in the tail annotated with $\emptyset$; 
\item[(ii)] there is a variable $y$ such that $T'$ passes through the right premiss of the rule $\mathsf{(case)}$ with the active variable $y$ infinitely many times.
\een 

By the definition of $\infty$-proof, the given branch of $\pi$ contains a tail $T$ satisfying conditions (a)--(g).  By condition (c), there is a variable $y$ such that the tail passes through the right premiss of the rule $\mathsf{(case)}$ with the active variable $y$ infinitely many times. If on $T$ the annotation $V$ is never empty, we put $T':=T$. Otherwise, there is a node $h$ of $T$ where $V=\emptyset$, then we can find the first application of the rule $\mathsf{(case)}$ above $h$ in $T$. By condition (a), the tail $T$ passes through the right premiss $\Delta$ of that rule application. We also see that the conclusion of this application is annotated by $\emptyset$ and the annotation of $\Delta$ is non-empty. We denote by $T'$ the end-part of $T$ starting from $\Delta$. Note that all annotations $V$ in $T'$ consist of the active variables of applications of the rule $\mathsf{(case)}$ in $T'$: all other rule applications do not increase the set $V$. Also, $T'$ trivially satisfies condition (ii) by (c). 

Next we show by induction along $T'$ that all annotations $V$ in $T'$ are non-empty. We have already established the basis of induction. The induction step for all the rules except $\mathsf{(\land),(\forall),(cut),(case)}$ is trivial ($V$ does not change). The induction step for the rules $\mathsf{(\land),(cut),(case)}$ follows from the $\Pi_{n+1}$-restriction of $\pi$, that is, by conditions (d)--(g). The annotation $V$ remains the same for $\mathsf{(\land),(cut)}$ and may only grow for $\mathsf{(case)}$. Finally, if a sequent is obtained by $\mathsf{(\forall)}$ with an active variable $z$ and a principal formula $\al{y}\phi$, then by (e) the formula $\phi[y\mapsto z]$ is $\Pi_{n+1}$. Moreover, $z\nin V$. Otherwise, there is an application of $\mathsf{(case)}$ with the active variable $z$ in $T'$ below this point, which contradicts condition (b). Hence, this rule application also preserves $V$. 

Thus, we have established that $T'$ satisfies conditions (i) and (ii), hence $\pi^\prime$ is an annotated $\infty$-proof. Trivially, $\pi^\prime$ is finite if $\pi$ is finite.

Now let us assume that the $\infty$-proof $\pi$ is regular. We show that $\pi^\prime$ is also regular by \emph{reductio ad absurdum}. Suppose there is an infinite sequence of pairwise non-isomorphic subtrees of $\pi^\prime$. Since $\pi^\prime$ is obtained from a regular $\infty$-proof, there are only finitely many non-isomorphic subtrees disregarding annotations. Therefore, there is a subsequence $(\mu_i)_{i\in \mathbb{N}}$ of the given sequence, where all members are isomorphic disregarding annotations. We see that the roots of $(\mu_i)_{i\in \mathbb{N}}$ are marked by non-identical annotated sequents obtained from a single unannotated sequent $\Delta$. However, any annotation occurring in $\pi^\prime$ is a subset of the union of $V$ and the finite set of variables that are active in applications of the inference rule $\mathsf{(case)}$ in $\pi^\prime$. Consequently, there can be only finitely many non-identical annotated sequents obtained from $ \Delta$ in $\pi^\prime$, which is a contradiction. We conclude that the annotated $\infty$-proof $\pi^\prime$ is regular.
\end{proof}

A \emph{cyclic annotated proof of the system $\mathsf{S}_n$} is a pair $(\kappa, d)$, where $\kappa$ is a finite tree of annotated sequents constructed in accordance with the annotated versions of inference rules of $\mathsf{S}_n$ and $d$ is a function with the following properties: the function $d$ is defined on the set of all leaves of $\kappa$ that are not marked by initial sequents; the image $d(a)$ of a leaf $a$ lies on the path from the root of $\kappa$ to the leaf $a$ and is not equal to $a$; $d(a)$ and $a$ are marked by the same sequents; all sequents on the path from $d(a)$ to $a$, including the ends, have the same non-empty annotation $V$; the path from $d(a)$ to $a$ intersects an application of the rule $\mathsf{(case)}$ on the right premiss. If the function $d$ is defined at a leaf $a$, then we say that nodes $a$ and $d(a)$ are \emph{connected by a back-link}. 

{For example, consider a cyclic annotated proof of the system $\mathsf{S}_n$
{\scriptsize\begin{gather*}\label{Example2}
\AXC{$\pi_0$}
\noLine
\UIC{$\vdots$}
\noLine
\UIC{$\overline{\varphi(0)}, \psi , \varphi (0) \upharpoonright \emptyset$}
\AXC{$\overline{\varphi(0)}, \psi ,\varphi (x)\upharpoonright \{x\} \tikzmark{A}$}
\LeftLabel{$\mathsf{(weak)}$}
\UIC{$\overline{\varphi(0)}, \psi , \varphi (x) , \varphi (s(x))\upharpoonright \{x\}$}
\AXC{$\pi_1$}
\noLine
\UIC{$\vdots$}
\noLine
\UIC{$\overline{\varphi(0)}, \psi ,\overline{\varphi (s(x)) }, \varphi (s(x))\upharpoonright \emptyset$}
\LeftLabel{$\mathsf{(\wedge)}$}
\BIC{$\overline{\varphi(0)}, \psi ,  \varphi (x) \wedge \overline{\varphi (s(x)) }, \varphi (s(x))\upharpoonright \{x\}$}
\LeftLabel{$\mathsf{(\exists)}$}
\UIC{$\overline{\varphi(0)}, \psi , \varphi (s(x)) \upharpoonright \{x\}$}
\LeftLabel{$\mathsf{(case)}$}
\RightLabel{ ,}
\BIC{$\overline{\varphi(0)},\psi , \varphi (x) \upharpoonright \{x\} \tikzmark{B}$}
\DisplayProof 
\begin{tikzpicture}[overlay,remember picture,>=latex,distance=8.0cm] \draw[->, thick] (pic cs:A) to [out=18,in=-10] (pic cs:B);
\end{tikzpicture}
\end{gather*}}
where $\varphi$ is an arbitrary $\Pi_{n+1}$-formula such that $x\in \mathit{
FV(\varphi)}$ and $\overline{\varphi (s(x))} \nin\Pi_{n+1}$, $\pi_0$ and $\pi_1$ are finite annotated $\infty$-proofs given by Lemma \ref{phiphi} and Lemma \ref{AnnLem}, and $\psi$ equals to $  \overline{\forall x \; (\varphi (x) \to \varphi (s(x)) )}$.}

Obviously, every cyclic annotated proof can be unravelled into a regular one. We prove the converse.
\begin{lemma}\label{CyclLem}
Any regular annotated $\infty$-proof of the system $\mathsf{S}_n$ can be obtained by unravelling of a cyclic annotated proof of the same system.
\end{lemma}
\begin{proof}
Assume we have a regular annotated $\infty$-proof $\pi$ of the system $\mathsf{S}_n$. Notice that each node $a$ of this tree determines the subtree $\pi_a$ with the root $a$. Let $m$ denote the number of non-isomorphic subtrees
of $\pi$. Consider any simple path $a_0, a_1, \dotsc , a_m$ in $\pi$ that starts at the root of $\pi$ and has length $m $. This path defines the sequence of subtrees $\pi_{a_0},\pi_{a_1}, \dotsc,\pi_{a_m}$. Since $\pi$ contains precisely $m$ non-isomorphic subtrees, the path contains a pair of different nodes $b$ and $c$ determining isomorphic subtrees $\pi_b$ and $\pi_c$. Without loss of generality, assume that $c$ is farther from the root than $b$. Note that the path from $b$ to $c$ intersects an application of the rule $\mathsf{(case)}$ on the right premiss and all sequents on this path have the same non-empty annotation $V$, because otherwise there is an infinite branch in $\pi$ that violates the corresponding condition on infinite branches of annotated $\infty$-proofs. We cut the path $a_0, a_1, \dotsc , a_m$ at the node $c$ and connect $c$, which has become a leaf, with $b$ by a back-link. By applying a similar operation to each of the remaining paths of length $m $ that start at the root, we ravel the regular annotated $\infty$-proof $\pi$ into the desired cyclic annotated one.
\end{proof}

\section{From cyclic proofs to ordinary ones} \label{uncycle}

In this section, we prove the converse to Corollary~\ref{sn}. 

\begin{theorem} \label{main1}
If $\mathsf{S}_n \vdash \Gamma$, then $\mathit{I \Sigma}_{n+1} \vdash \bigvee \Gamma$.
\end{theorem}
\begin{proof}
By Lemma \ref{AnnLem} and Lemma \ref{CyclLem}, it is sufficient to prove that $\mathit{I \Sigma}_{n+1} \vdash \bigvee \Gamma$ whenever there exists a cyclic annotated proof $\pi = (\kappa, d)$ of a sequent $\Gamma \upharpoonright V$ in the system $\mathsf{S}_n$ for some finite set of individual variables $V$. We proceed by induction on the number of nodes in $\pi$ (we call this \emph{main induction} in what follows). 

Case 1. Suppose that there are no leaves of $\pi$ connected by back-links with the root. If $\pi$ consists of a single node, then $\Gamma \upharpoonright V$ has the form $\Delta, t_0 = t_1, t_0 \neq t_1 \upharpoonright V$ or $\Delta, s(t) \neq 0 \upharpoonright V$. Therefore, we obtain $\mathit{I \Sigma}_{n+1} \vdash \bigvee \Gamma$. 
Otherwise, consider the last application of an inference rule in $\pi$. 

Since the subcases of all inference rules are similar, we consider only the subcase of the rule $\mathsf{(case)}$. Suppose $\pi$ has the form
\[
\AXC{$\pi^\prime$}
\noLine
\UIC{$\vdots$}
\noLine
\UIC{$\Gamma[y \mapsto 0] \upharpoonright \emptyset$}
\AXC{$\pi^{\prime\prime}$}
\noLine
\UIC{$\vdots$}
\noLine
\UIC{$\Gamma[y \mapsto s(y)] \upharpoonright V^{\prime\prime}$}
\LeftLabel{$\mathsf{case}$}
\RightLabel{$(y \in FV(\Gamma))$, }
%\RightLabel{ ,}
\BIC{$\Gamma \upharpoonright V$}
%\BIC{$\Gamma $}
\DisplayProof
\]
where $\pi^\prime$ and $\pi^{\prime\prime}$ are cyclic annotated proofs in $\mathsf{S}_n$. Applying the induction
hypothesis for $\pi^\prime$ and $\pi^{\prime\prime}$, we obtain $\mathit{I \Sigma}_{n+1} \vdash \bigvee \Gamma[y \mapsto 0]$ and $\mathit{I \Sigma}_{n+1} \vdash \bigvee \Gamma[y \mapsto s(y)]$. Consequently, $\mathit{I \Sigma}_{n+1} \vdash \bigvee \Gamma$ since $\mathit{I \Sigma}_{n+1} \vdash \forall y\; (y = 0 \vee \exists z\; y = s(z))$.

Case 2. Suppose that there is a leaf of $\pi = (\kappa, d)$ connected by a back-link with the
root. In this case, all sequents on the path from the root to the leaf have the same non-empty annotation $V$.

In what follows, a \emph{directed path} is a path along the edges of $\kappa$, directed away from the root, and back-links.
Let $M$ denote the following set of nodes of $\pi$: $a \in M$ if and only if there is a finite directed path from $a$ to the root of $\pi$. Note that, for any $a \in M$, the sequent of the node $a$ has the form $\Phi_a, \Psi_a  \upharpoonright V$, where $\Psi_a$ is the multiset of all $\Pi_{n+1}$-formulas of the sequent. Since $V$ is the same for all $a\in M$ we consider it fixed for the rest of the proof. We put $\Gamma_a \coloneq \Phi_a, \Psi_a$ and $\varphi_a \coloneq \overline{\bigvee \Phi_a}$, $\psi_a \coloneq \bigvee \Psi_a$. In addition, we denote the root of $\pi$ by $r(\pi)$.

We proceed with the following lemma.

\begin{lemma}\label{side_equiv}
For any two nodes $a,b\in M$, the formulas $\varphi_a$ and $\varphi_b$ are equivalent in predicate logic:
\begin{gather*}
\mathit{\vdash \varphi_a \leftrightarrow \varphi_b.}
\end{gather*}
\end{lemma}
\begin{proof}[Proof of Lemma]
It is sufficient to prove that $ \vdash \varphi_a \rightarrow \varphi_b$ as $a, b$ are arbitrary elements of $M$. By the definitions of $\varphi_a,  \varphi_b$ and the principle of contraposition, it is sufficient to demonstrate that $\vdash \bigvee \Phi_b \rightarrow \bigvee \Phi_a$.

Observe that a finite directed path exists between $a$ and $b$, since there is a finite directed path from $a$ to the root $r(\pi)$, as well as a finite directed path from the root to any other node in $M$. We continue the argument by induction on the length of the shortest directed path between $a$ and $b$.

\textsc{Basis of induction:} If $a = b$, then $\Phi_a = \Phi_b$, hence the result trivially holds.

\textsc{Induction step:} Let us denote by $a'$ the next node on the directed path from $a$ to $b$. 

If $a$ is a leaf of $\pi$ and $d(a) = a^\prime$, then $\Phi_a = \Phi_{a^\prime}$. By the induction hypothesis, we also have $\vdash \bigvee \Phi_b \rightarrow \bigvee \Phi_{a^\prime}$. Hence,  $\vdash \bigvee \Phi_b \rightarrow \bigvee \Phi_a$.

Otherwise, the sequent at $a$ is derived by applying some inference rule to the sequent of node $a^\prime$ and possibly some other sequent. 
%Without loss of generality we can assume that the sequent $\Gamma_{a^\prime}$ is the right premiss of the rule in case of any inference rule with two premiss.
Clearly, $a^\prime$ is an element of $M$, hence $a^\prime$ is annotated with $V$.
By the induction hypothesis we have $\vdash \bigvee \Phi_b \rightarrow \bigvee \Phi_{a'}$, therefore it is sufficient to demonstrate that 
\begin{equation*}\label{claim} \vdash \bigvee \Phi_{a^\prime} \rightarrow \bigvee \Phi_a.
\end{equation*}
We prove this claim by analyzing various cases according to different inference rules.

Case A. Most of the rules do not change the $\Phi$ part of the sequent. This holds for the rules $\mathsf{(ref)}$, $\mathsf{(rep)}$, $\mathsf{(add_0)}$, $\mathsf{(add_s)}$, $\mathsf{(mult_0)}$, $\mathsf{(mult_s)}$, $\mathsf{(pred)}$:  In all these cases $\Gamma_a$ and $\Gamma_{a'}$ can only differ in atomic formulas.  

Case B. If $\Gamma_a$ is obtained from $\Gamma_{a'}$ by $\mathsf{(weak)}$, then $\Phi_{a'}\subseteq \Phi_a$ and the claim trivially follows. 

Case C. Suppose $\Gamma_{a^\prime}$ is a premiss of an application of the rule ($\mathsf{\vee}$) and $\Gamma_a$ is its conclusion. Let $\alpha \vee \beta$ be its principal formula. If $\alpha \vee \beta \nin \Pi_{n+1}$, then $\alpha\nin \Pi_{n+1} $ or $\beta\nin \Pi_{n+1} $. We have $\vdash \bigvee \Phi_{a^\prime} \rightarrow \bigvee \Phi_a$ since $\vdash\alpha\to \alpha \vee \beta$ and $\vdash\beta\to \alpha \vee \beta$. Otherwise, $\alpha \vee \beta$ is $\Pi_{n+1}$ and so are $\alpha$ and $\beta$ (by the definition of $\Pi_{n+1}$), hence $\Phi_{a^\prime} = \Phi_a$.

Case D. Suppose $\Gamma_a$ is obtained from $\Gamma_{a'}$ by $\mathsf{(\exists)}$. If $\phi[y\mapsto t]$ is $\Pi_{n+1}$ then $\Phi_{a'}=\Phi_a$. Otherwise, both $\phi[y\mapsto t]$ and $\ex{y}\phi$ belong to $\Phi_{a'}$. We observe that $\vdash \phi[y\mapsto t]\to \ex{y}\phi$, hence  $\vdash \bigvee \Phi_{a^\prime} \rightarrow \bigvee \Phi_a$. 

The remaining rules are handled somewhat similarly by using the annotation conditions. 

Case E. Suppose $\Gamma_a$ is obtained from $\Gamma_{a^\prime}$ by $\mathsf{(\forall)}$. As these sequents have the same non-empty annotation $V$, the principal formula has to be $\Pi_{n+1}$. Hence, in this case $\Phi_{a^\prime} = \Phi_a$.

Case F. Suppose $\Gamma_{a^\prime}$ is a premiss of an application of $\mathsf{(\wedge)}$ and $\Gamma_a$ is its conclusion. As these sequents have the same non-empty annotation $V$, $\Phi_{a^\prime} \subseteq \Phi_a$ and the claim follows.

Case G. Suppose $\Gamma_{a^\prime}$ is a premiss an application of  $\mathsf{(cut)}$ and $\Gamma_a$ is its conclusion. As these sequents have the same non-empty annotation $V$ the cut formula has to be $\Pi_{n+1}$, therefore $\Phi_{a^\prime} = \Phi_a$.

Case H. Suppose $\Gamma_{a^\prime}$ is the right premiss of an application of  $\mathsf{(case)}$ and $\Gamma_a$ is its conclusion. As the annotation $V$ of these sequents is non-empty, we conclude that the active variable in the application of the rule $\mathsf{(case)}$ is permitted to occur freely only within the $\Pi_{n+1}$-formulas in $\Gamma_a$. It follows that $\Phi_{a^\prime} = \Phi_a$.

This concludes the proof of Lemma~\ref{side_equiv}.
\end{proof}

%Note that, for an application of the rule ($\forall$) in $\pi$, the conclusion of the rule belongs to $M$ if and only if   its premiss does so. 
Consider the set of all applications of the rule ($\mathsf{\forall}$) in $\pi$ whose conclusions belong to $M$. By $\mathit{B}$, we denote the set of active variables of these applications.
%It is important to note that all variables from $\mathit{B}$ may occur freely exclusively within $\Pi_{n+1}$-formulas, implying that they do not appear freely in the formulas $\varphi_a$ for $a \in M$.
We also set 
\[\theta \coloneq {\forall}_\mathit{B} \bigwedge_{c \in C} \psi_c,\]
where ${\forall}_\mathit{B}$ means the finite sequence of universal quantifiers for variables of $\mathit{B}$ and $C$ is the set of conclusions of applications of the rule $\mathsf{(case)}$ in $M$. Note that $\theta$ belongs to $\Pi_{n+1}$. 

For any $a \in M$, we define its rank $\mathit{rk} (a)$ as the length of the longest directed path from $a$ to a node $c$ from $C$ such that only the last node of the path belongs to $C$. Note that $\mathit{rk} (c) = 0$ if $c \in C$. The notion of rank is well-defined since any infinite directed path in $\pi$ must intersect applications of the rule $\mathsf{(case)}$. The following lemma demonstrates that all sequents from $M$  are implied in $\mathit{I \Sigma_{n+1}}$ by a single formula $\theta$.

\begin{lemma}\label{theta_gamma}
For any node $a \in M$, it holds that
\begin{gather*}
\mathit{I \Sigma}_{n+1} \vdash \theta \rightarrow (\varphi_a \to \psi_a). 
\end{gather*}
\end{lemma}
\begin{proof}[Proof of Lemma]
Observe that this is equivalent to $\mathit{I \Sigma}_{n+1} \vdash \theta \rightarrow \bigvee \Gamma_a$. We prove the lemma by induction on $\mathit{rk}(a)$.

Case A. Should $a$ be an element of $C$, the claim holds trivially, as $I \Sigma_{n+1} \vdash \theta \rightarrow \psi_a$. If $a$ is a leaf of $\pi$ and $d(a) = b$, then the claim follows from the induction hypothesis since $\Gamma_a=\Gamma_b$.

Assume that the sequent at the node $a$ results from an application of an inference rule to the sequent at the node $a^\prime \in M$ and possibly some other sequent, and $\mathit{rk}(a^\prime) < \mathit{rk}(a)$. Note that this rule is not $\mathsf{(case)}$ because $a \nin C$. Without loss of generality, we can assume that $\Gamma_{a^\prime}$ is the left premiss of the rule in case of any inference rule with two premisses.

Case B. Suppose the sequent $\Gamma_a$ is derived by an application of the rule ($\mathsf{\wedge}$). The proof fragment at node $a$ thus has the form 
\[
\AXC{$\kappa^\prime$}
\noLine
\UIC{$\vdots$}
\noLine
\UIC{$\Delta, \alpha \upharpoonright V$}
\AXC{$\kappa^{\prime\prime}$}
\noLine
\UIC{$\vdots$}
\noLine
\UIC{$\Delta, \beta \upharpoonright V^{\prime\prime} $}
\LeftLabel{$\mathsf{(\wedge)}$}
\RightLabel{ .}
%\RightLabel{ ,}
\BIC{$\Delta, \alpha \wedge \beta \upharpoonright V$}
%\BIC{$\Gamma $}
\DisplayProof
\] 
Observe that $\mathit{I \Sigma}_{n+1} \vdash \theta \rightarrow \bigvee \Gamma_{a^\prime}$ holds by the induction hypothesis, hence $\mathit{I \Sigma}_{n+1} \vdash \theta \rightarrow \bigvee \Delta \vee \alpha$. If the node corresponding to the right premiss of the rule is an element of $M$, then by analogous reasoning, $\mathit{I \Sigma}_{n+1} \vdash \theta \rightarrow \bigvee \Delta \vee \beta$ holds. Otherwise, no back-links exist from the nodes of the tree $\kappa^{\prime\prime}$ to those within $M$. The proof $\kappa''$ is a cyclic proof with less nodes than $\pi$, consequently by the \emph{main induction hypothesis} $\mathit{I \Sigma}_{n+1} \vdash \bigvee \Delta \vee \beta$, therefore  $\mathit{I \Sigma}_{n+1} \vdash \theta \rightarrow \bigvee \Delta \vee \beta$. As a result, $\mathit{I \Sigma}_{n+1} \vdash \theta \rightarrow \Delta \vee (\alpha \wedge \beta)$ follows.

Case C. The case of rule $\mathsf{(cut)}$ is similar to the case of $\mathsf{(\wedge)}$.

Case D. Suppose $\Gamma_a$ is the conclusion of an application of rule $\mathsf{(\forall)}$. The corresponding fragment of the proof at node $a$ has the form 
\[
\AXC{$\kappa^\prime$}
\noLine
\UIC{$\vdots$}
\noLine
\UIC{$\Delta, \alpha[y \mapsto z] \upharpoonright V$}
\LeftLabel{$\mathsf{(\forall)}$}
\RightLabel{($z \nin \mathit{FV}(\Delta, \forall y\; \alpha))$.}
%\RightLabel{ ,}
\UIC{$\Delta, \forall y\; \alpha \upharpoonright V$}
%\BIC{$\Gamma $}
\DisplayProof
\] 
The induction hypothesis $\mathit{I \Sigma}_{n+1} \vdash \theta \rightarrow \bigvee \Gamma_{a^\prime}$ means $$\mathit{I \Sigma}_{n+1} \vdash \theta \rightarrow \bigvee \Delta \vee \alpha[y \mapsto z].$$ Observe that $z$ is an element of $B$, thereby $z \nin \mathit{FV}(\theta)$. Furthermore, $z \nin \mathit{FV}(\Delta, \forall y\; \alpha)$ by the definition of $\mathsf{(\forall)}$. Hence, by predicate logic we obtain $\mathit{I \Sigma}_{n+1} \vdash \theta \rightarrow \bigvee \Delta \vee \forall y\; \alpha$.

\textsc{Case E.} For any other rule, it is evident that $\mathit{I \Sigma}_{n+1} \vdash \bigvee \Gamma_{a^\prime} \rightarrow \bigvee \Gamma_a$. By the induction hypothesis $\mathit{I \Sigma}_{n+1} \vdash \theta \rightarrow \bigvee \Gamma_{a^\prime}$, hence $\mathit{I \Sigma}_{n+1} \vdash \theta \rightarrow \bigvee \Gamma_a$ and we are done.
\end{proof}

Let $y_1, \dotsc, y_m$ be the list of all active variables of applications of the rule $\mathsf{(case)}$ whose conclusions belong to $M$. We set 
\[\zeta(z) \coloneq \forall y_1\leq z \dotsc \forall y_m\leq z \; (y_1 + \dotsb + y_m = z \rightarrow \theta),\]
where $z \nin \mathit{FV}(\theta)$, $z\notin B$  and $z \nin \mathit{FV}(\varphi_{r(\pi)})$. Note that $\zeta(z)$ is a $\Pi_{n+1}$-formula.

%Now we show that induction base case and step holds for $\zeta(z)$, if one assumes $\varphi_{r(\pi)}$ as a premiss. 
\begin{lemma}\label{zeta_ind}
It holds that
\begin{gather*}
\mathit{I \Sigma}_{n+1} \vdash \varphi_{r(\pi)} \rightarrow \zeta(0), \qquad \mathit{I \Sigma}_{n+1} \vdash \varphi_{r(\pi)} \rightarrow \forall z\; (\zeta(z) \rightarrow \zeta(s(z))).
\end{gather*}
\end{lemma}
\begin{proof}[Proof of Lemma]
For each $c$ from $C$, the corresponding fragment of the proof has the form
\[
    \AXC{$\kappa^\prime$}
    \noLine
    \UIC{$\vdots$}
    \noLine
    \UIC{$\Phi_c, \Psi_c[y_j \mapsto 0] \upharpoonright \emptyset$}
    \AXC{$\kappa^{\prime\prime}$}
    \noLine
    \UIC{$\vdots$}
    \noLine
    \UIC{$\Phi_c, \Psi_c[y_j \mapsto s(y_j)] \upharpoonright V$}
    \LeftLabel{$\mathsf{(case)}$}
    \RightLabel{($y_j \in FV(\Phi_c, \Psi_c))$ ,}
    \BIC{$\Phi_c, \Psi_c \upharpoonright V$}
    \DisplayProof
\]
where $y_j$ is the active variable of the inference. Note, that $y_j \nin \mathit{FV}(\Phi_c)$ as $V$ is non-empty and in this case active variable of the rule $\mathsf{(case)}$ occurs freely only in $\Pi_{n+1}$-formulas of conclusion. Since there are no left premisses of the rule $\mathsf{(case)}$ in between any two nodes connected by a back-link, we have an annotated cyclic  proof $\pi^\prime =(\kappa^\prime, d^\prime)$ of $\Phi_c,\Psi_c[y_j \mapsto 0] \upharpoonright \emptyset$ in $\mathsf{S}_n$. From the main induction hypothesis for $\pi^\prime$, we obtain $\mathit{I \Sigma}_{n+1} \vdash \varphi_c \rightarrow \psi_c[y_j \mapsto 0]$. Applying Lemma \ref{side_equiv}, we get $\mathit{I \Sigma}_{n+1} \vdash \varphi_c \leftrightarrow \varphi_{r(\pi)}$ and $\mathit{I \Sigma}_{n+1} \vdash \varphi_{r(\pi)} \rightarrow \psi_c[y_j \mapsto 0]$. As $y_j \nin FV(\Phi_c)$, we conclude that $\varphi_{r(\pi)}$ does not depend on $y_j$:
$$\mathit{I\Sigma}_{n+1}\vdash \phi_{r(\pi)}[y_j\mapsto 0]\eqv \phi_{r(\pi)}.$$
The same holds for every $j \in \{1, \dots, m\}$, therefore we obtain $$\mathit{I \Sigma}_{n+1} \vdash \varphi_{r(\pi)} \rightarrow \psi_c[y_1 \mapsto 0, \dotsc, y_m \mapsto 0].$$
Consequently, 
\begin{equation}\label{ber} \mathit{I \Sigma}_{n+1} \vdash \varphi_{r(\pi)} \rightarrow \bigwedge_{c \in C} \psi_c[y_1 \mapsto 0, \dotsc, y_m \mapsto 0].
\end{equation}

Now we argue in a similar manner that $\varphi_{r(\pi)}$ does not depend on the variables from $B$. Let $x_i\in B$ be the active variable in an application of $\mathsf{(\forall)}$ whose conclusion $\Gamma_a = \Phi_a,\Psi_a$ is in $M$. Since $V$ is non-empty, its principal formula $\al{x}\psi\in \Pi_{n+1}$ does not occur in $\Phi_a$, so  $x_i\notin\mathit{FV}(\Phi_a)$. By Lemma \ref{side_equiv}, we obtain that $\varphi_{r(\pi)}$ is equivalent to the formula $\varphi_a$ in which $x_i$ does not occur, and the claim follows. 

Now, applying Bernays' rule to \refeq{ber} we obtain $$\mathit{I \Sigma}_{n+1} \vdash \varphi_{r(\pi)} \rightarrow \theta[y_1 \mapsto 0, \dotsc, y_m \mapsto 0],$$ and therefore $\mathit{I \Sigma}_{n+1} \vdash \varphi_{r(\pi)} \rightarrow \zeta(0).$

In order to prove $\mathit{I \Sigma}_{n+1} \vdash \varphi_{r(\pi)} \rightarrow \forall z\; (\zeta(z) \rightarrow \zeta(s(z)))$, it is sufficient to show 
\[\mathit{I \Sigma}_{n+1} \vdash \varphi_{r(\pi)} \rightarrow (\zeta(z) \rightarrow (y_1 + \dotsb + y_n = s(z) \rightarrow \psi_c)),\]
for each $c$ from $C$, because $\varphi_{r(\pi)}$ does not depend on variables from $B \cup \{z, y_1, \dotsc, y_m\}$. Let $y_j$ be the active variable in an application of the rule $\mathsf{(case)}$ at the node $c$. Arguing in $\mathit{I \Sigma}_{n+1}$, we consider two cases: $y_j = 0$ or $y_j = s(y^\prime_j)$. If $y_j = 0$ then $\varphi_{r(\pi)} \rightarrow \psi_c[y_j \mapsto 0]$ is provable in $\mathit{I \Sigma_{n+1}}$ as we have seen.

Reasoning in $\mathit{I\Sigma}_{n+1}$ assume $y_j = s(y^\prime_j)$, $\varphi_{r(\pi)}$, $\zeta(z)$ and 
\begin{gather*}
y_1 + \dotsb + y_{j-1} + s(y^\prime_j) + y_{j+1} + \dotsb + y_n = s(z).
\end{gather*}
Then
\begin{gather}\label{assertion1}
y_1 + \dotsb + y_{j-1} + y^\prime_j + y_{j+1} + \dotsb + y_n = z.
\end{gather}
From (\ref{assertion1}) and $\zeta(z)$, we obtain $\theta[y_j \mapsto y^\prime_j]$. 

Recall that the node $c$ is the conclusion of an application the rule $\mathsf{(case)}$. Let $b$ be the node corresponding to the right premiss of this application. From Lemma \ref{theta_gamma} and $\theta[y_j \mapsto y^\prime_j]$, we have $\varphi_b \rightarrow \psi_b[y_j \mapsto y^\prime_j]$, as $y_j \nin \mathit{FV}(\varphi_b)$. Applying Lemma \ref{side_equiv} and $\varphi_{r(\pi)}$, we obtain $\psi_b[y_j \mapsto y^\prime_j]$. As $b$ is the right premiss of the rule $\mathsf{(case)}$ for the node $c$, we have $\psi_b[y_j \mapsto y^\prime_j] = \psi_c[y_j \mapsto s(y^\prime_j)]$. We recall that $y_j = s(y^\prime_j)$ and obtain the required formula $\psi_c$. The second case is checked, and the lemma is proven.
\end{proof}

Recall that the induction principle for $\Pi_{n+1}$-formulas is provable in $\mathit{I \Sigma}_{n+1}$, see \cite{HP}. Therefore, we derive $\mathit{I \Sigma}_{n+1} \vdash \varphi_{r(\pi)} \rightarrow  \forall z\; \zeta(z)$ from Lemma \ref{zeta_ind}. 
%Renaming bound variables in $\zeta(z)$ and substituting the term $y_1 + \dotsb + y_m$ for $z$, we obtain 
%\begin{align*}
%\mathit{I \Sigma}_{n+1} &\vdash \varphi_{r(\pi)} \rightarrow \forall z\; \forall y^{\prime\prime}_1,\dotsc, \forall y^{\prime\prime}_m \; (y^{\prime\prime}_1 + \dotsb + y^{\prime\prime}_m = z \rightarrow \theta(y^{\prime\prime}_1, \dotsc, y^{\prime\prime}_m)) \\
%&\vdash \varphi_{r(\pi)} \rightarrow \forall y^{\prime\prime}_1, \dotsc, \forall y^{\prime\prime}_m \; (y^{\prime\prime}_1 +\dotsb + y^{\prime\prime}_m = y_1 + \dotsb + y_m \rightarrow \theta(y^{\prime\prime}_1,\dotsc, y^{\prime\prime}_m)) \\
%&\vdash \varphi_{r(\pi)} \rightarrow (y_1 + \dotsb + y_m = y_1 + \dotsb + y_m \rightarrow \theta) \\
%&\vdash \varphi_{r(\pi)} \rightarrow \theta.
%\end{align*}

The formula $\al{z}\zeta(z)$ clearly implies $\al{y_1}\dots \al{y_m}\theta(y_1,\dots,y_m)$ in $\mathit{I\Sigma}_0$, so we obtain 
$$I\Sigma_{n+1}\vdash \varphi_{r(\pi)}\to \theta.$$
By Lemma \ref{theta_gamma}, we also have $\mathit{I \Sigma}_{n+1} \vdash \varphi_{r(\pi)} \rightarrow (\theta \rightarrow \psi_{r(\pi)})$.
Therefore, $\mathit{I \Sigma}_{n+1} \vdash \varphi_{r(\pi)} \rightarrow \psi_{r(\pi)}$.
Since $\mathit{I \Sigma}_{n+1} \vdash \bigvee \Gamma \leftrightarrow (\varphi_{r(\pi)} \rightarrow \psi_{r(\pi)})$, we conclude that $\mathit{I \Sigma}_{n+1} \vdash \bigvee \Gamma$.
\end{proof}

Thus, we have established that the system $\mathsf{S}_n$ is deductively equivalent to $\mathit{I\Sigma}_{n+1}$. As a corollary we also obtain that unrestricted regular $\infty$-proofs correspond to proofs in Peano arithmetic. 
\bcor A formula $\phi$ is provable in $\PA$ iff $\phi$ has a regular $\infty$-proof. 
\ecor 
\begin{proof} Any regular $\infty$-proof is $\Pi_{n+1}$-restricted, for some $n$. 
\end{proof}

\section{Proof systems for induction rules} \label{rules}
Fragments of Peano arithmetic defined by restricted forms of induction rules are well-studied in proof theory, see~\cite{Bek97a,Jer2020} for detailed surveys. Here we are interested in the axiomatization of such theories using non-well-founded and cyclic proofs. 

First, we introduce a natural modification of $\mathsf{S}_n$ that yields a system deductively equivalent to $\mathit{I\Pi}_{n+1}^R$, that is, to the closure of $\Q$ under the $\Pi_{n+1}$-induction rule. More generally, given an extension $T$ of $\Q$ by some set of additional axioms, we would like to characterize the closure of $T$ under the $\Pi_{n+1}$-induction rule. To this end, we introduce proofs from assumptions.   

Given a set $T$ of sentences, an \emph{$\infty$-proof from assumptions $T$} is an $\infty$-proof in which the rules of the calculus $\SQ$ are extended by the initial sequents of the form $\{\varphi\}$ for all $\phi\in T$. %Notice that the assumptions never occur in the infinite branches of the proof tree, therefore the notion of $\Pi_n$-restricted $\infty$-proof is essentially unaffected by the presence of assumptions.
As before, a \emph{regular $\infty$-proof from assumptions} is an $\infty$-proof from assumptions such that the proof tree has only finitely many non-isomorphic subtrees. Note also that assumptions never occur in infinite branches of the proof tree. Therefore, the notion of $\Pi_{n}$-restricted $\infty$-proof is essentially unaffected by the presence of assumptions.

We say that a $\Pi_{n+1}$-restricted $\infty$-proof (from assumptions $T$) is an \emph{$\infty$-proof of $\SR_n$ (from $T$)} whenever every infinite branch in it has a tail satisfying, instead of condition (g), the following stronger condition: if the tail passes through the right premiss $\Gamma [z\mapsto s(z)]$ of the rule $\mathsf{(case)}$, then $\Gamma$ consists entirely of $\Pi_{n+1}$-formulas. Analogously to the definition of $\mathsf{S}_n$, a sequent $\Gamma$ is \emph{provable in $\SR_n$ (from assumptions $T$)} if there is a regular $\infty$-proof of $\SR_n$ (from $T$) with the root marked by $\Gamma $. For this calculus, the notions of annotated $\infty$-proof and cyclic annotated proof of $\mathsf{S}_n$ are modified appropriately.

% Annotated proofs from assumptions and cyclic proofs from assumptions are introduced similarly.  

%We say that an $\mathsf{S}_n$-proof $\pi$ (possibly from assumptions) is an \emph{$\SR_n$-proof} if, for every application of the rule $\mathsf{(case)}$ in $\pi$, its premisses (and conclusion) consist entirely of $\Pi_{n+1}$-formulas. This definition applies equally to $\infty$-proofs and to cyclic proofs of $\mathsf{S}_n$. 

\bl \label{sr} The set of formulas provable in $\SR_n$ (from assumptions $T$) is closed under the induction rule $(\mathrm{IR})$ for $\Pi_{n+1}$-formulas.
\el 

\bp Assume $\phi(x)$ is $\Pi_{n+1}$, $x\in \mathit{FV}(\varphi)$ and both formulas $\phi(0)$ and $\al{x}(\phi(x)\to \phi(s(x)))$ are provable in $\SR_n$. Then the sequents $\{\phi(0)\}$ and $\{\overline{\varphi(x)} , \varphi (s(x))\}$ have regular $\infty$-proofs of $\SR_n$ denoted $\sigma_0$ and $\sigma_1$, respectively. We combine them into the following regular $\infty$-proof $\pi$:    
\begin{equation*}\label{exrule}
\AXC{$\sigma_0$}
\noLine
\UIC{$\vdots$}
\noLine
\UIC{$\varphi(0)$}
\AXC{$\pi$}
\noLine
\UIC{$\vdots$}
\noLine
\UIC{$\varphi(x)$}
\LeftLabel{$\mathsf{(weak)}$}
\UIC{$\varphi(x), \varphi (s(x))$}
\AXC{$\sigma_1$}
\noLine
\UIC{$\vdots$}
\noLine
\UIC{$\overline{\varphi(x)} , \varphi (s(x))$}
\LeftLabel{$\mathsf{(cut)}$}
\BIC{$\varphi (s(x))$}
\LeftLabel{$\mathsf{(case)}$}
\BIC{$\varphi (x)$}
\DisplayProof 
%\begin{tikzpicture}[overlay,remember picture,>=latex,distance=-4.2cm] \draw[->, thick] (pic cs:A) to [out=-26,in=10] (pic cs:B);
%\end{tikzpicture}
\end{equation*}

\medskip
Here, the subtree $\pi$ is isomorphic to the whole $\infty$-proof. It is easy to check that the displayed proof is a regular $\infty$-proof of $\SR_n$. 
%Here, annotation $V$ is $\emptyset$ or $\{x\}$ depending on whether $\overline{\phi(x)}$ belongs to $\Pi_{n+1}$.
\ep 

The converse to Lemma \ref{sr} also holds. Let $T+\pir{n}$ denote the closure of $T$ under the induction rule for $\Pi_{n}$ formulas.

\begin{theorem} \label{main2}
If $\Gamma$ is provable in $\SR_n$ from assumptions $T$, then $\Gamma$ is a theorem of $\Q+T+\pir{n+1}$. 
\end{theorem}
\begin{proof}
The proof is almost the same as that of Theorem~\ref{main1}. Read everywhere in that proof $\Q+T+\pir{n+1}$ instead of $\mathit{I\Sigma}_{n+1}$. Observe that Lemma~\ref{side_equiv} stays the same. The proof of Lemma \ref{theta_gamma} does not change except for the reference to the now stronger main induction hypothesis. Lemma~\ref{zeta_ind} also goes without change. 

Now we arrive at the final part of the proof. There, the crucial point is to apply the $\Pi_{n+1}$-induction rule with a side formula $\phi_{r(\pi)}$. By Lemma~\ref{side_equiv} we observe that $\phi_{r(\pi)}$ is logically equivalent to $\phi_a$, where $a$ is the conclusion of some $\mathsf{(case)}$ rule application in $M$. However, $\phi_a=\overline{\bigvee\Phi_a}$ where $\Phi_a=\emptyset$ by our definition of $\SR_n$. Hence, $\phi_{r(\pi)}$ must be logically provable.  
It follows that $\Q+T+\pir{n}$ proves $\al{z}\zeta(z)$ and we finish the argument as before. 
\end{proof} 

We conclude that the system $\SR_n$ exactly axiomatizes the fragment $\mathit{I\Pi}^R_{n+1}$ of Peano arithmetic axiomatized by the $\Pi_{n+1}$-induction rule over $\Q$. By a well-known result of Parsons~\cite{Par68}, for $n>0$, theories $\mathit{I\Pi}^R_{n+1}$ and $\mathit{I\Sigma}_n^R$ are deductively equivalent. However, over an arbitrary extension $T$ of $\Q$, this need not be so: For example, $\mathit{I\Sigma}_n$ is closed under the $\Sigma_n$-induction rule, but not under the $\Pi_{n+1}$-induction rule, see \cite{Bek97a} for exact characterizations. Also, $n=0$ is an exception: We do not know if  $\mathit{I\Pi}^R_1$ is strictly stronger than $\mathit{I\Sigma}^R_0\equiv \mathit{I\Delta}_0$, although this holds over the elementary arithmetic $\EA$.  

In order to formulate a version of $\mathsf{S}_n$ closely related to the $\Sigma_n$-induction rule, we need to modify the definition of $\Pi_n$-restricted $\infty$-proof as follows. 

\begin{definition}
An $\infty$-proof (possibly from assumptions) is
\emph{$\Sigma_n$-restricted } if every infinite branch in it has a tail satisfying conditions (a)--(c) from the definition of $\infty$-proof and the following ones: 
(d') if the tail passes through the left (right) premiss of the rule ($\wedge$), then the formula $\varphi$ ($\psi$) belongs to $\Sigma_{n}$, (e') 
the tail does not intersect applications of the rule ($\forall$), (f') if the tail passes through the left (right) premiss of the rule $\mathsf{(cut)}$, then the cut formula $\varphi$ ($\overline{\varphi}$) is $\Sigma_{n}$,
(g') if the tail passes through the right premiss $\Gamma [z\mapsto s(z)]$ of the rule $\mathsf{(case)}$, then $\Gamma$ consists entirely of $\Sigma_n$-formulas. 
\end{definition}  
Since $\Sigma_n\subseteq \Pi_{n+1}$, we see that each $\Sigma_n$-restricted $\infty$-proof is also $\Pi_{n+1}$-restricted. However, due to condition (e'), a $\Pi_{n}$-restricted $\infty$-proof need not necessarily be $\Sigma_{n+1}$-restricted. 
Now we say that a sequent $\Gamma$ is \emph{provable in $\ST_n$ (from assumptions $T$)} if there is a $\Sigma_n$-restricted regular $\infty$-proof (from $T$) with the root marked by $\Gamma$. We also naturally modify the notions of annotated $\infty$-proof and cyclic annotated proof of $\mathsf{S}_n$ in order to obtain the corresponding notions for $\ST_n$.

We notice that analogs of Lemma \ref{sr} and Theorem \ref{main2} hold for $\ST_n$ with similar proofs.

\bl \label{sr1} The set of formulas provable in $\ST_n$ (from assumptions $T$) is closed under the induction rule $(\mathrm{IR})$ for $\Sigma_{n}$-formulas.
\el 

\begin{theorem} \label{main3}
If $\Gamma$ is provable in $\ST_n$ from assumptions $T$, then $\Gamma$ is a theorem of $\Q+T+\sir{n}$.
\end{theorem}
\begin{proof}
We only remark that in this case the formulas $\Psi_a$ are $\Sigma_n$. Due to condition (e') the set of variables $B$ is empty, hence the formula $\theta$ is $\bigwedge_{c\in C}\psi_c$. Hence, $\theta\in\Sigma_n$ and $\zeta$ is obtained from a $\Sigma_n$-formula by bounded universal quantifiers. As in Theorem \ref{main2}, the final part of the proof boils down to an application of the induction rule (without side formulas) for $\zeta$. For $n=0$ such an inference is clearly admissible by $\sir{0}$. 

For $n>0$, the following lemma shows that such an application is also reducible to some  applications of $\sir{n}$. We state it for one bounded quantifier, but the proof clearly extends to any number of them.   

\begin{lemma} \label{b_induction}
Let $T$ be an extension of $\Q$ and $\psi(z)$ be a formula of the form $\al{x\leq z}\ex{y}\phi(x,y,z)$ with $\phi\in\Pi_{n}$. Assume $T\vdash \psi(0)$ and $T\vdash \al{z}(\psi(z)\to \psi(s(z)))$. Then $T+\sir{n+1}\vdash \al{z}\psi(z).$
\end{lemma}

\begin{proof} Without loss of generality we may assume $T$ to be an extension of $\mathit{I\Delta}_0$, since $\mathit{I\Delta}_0$ is contained in (and is equivalent to) $\mathit{I\Sigma}_0^R$. 
Now consider the formula $$\phi'(z,v):=\al{x\leq z}\ex{y\leq v}\phi(x,y,z).$$ 
Notice that $\phi'(z,v)$ is equivalent to a $\Pi_n$ formula, say $\phi''(z,v)$, in $\mathit{I\Sigma}_n$, since $\mathit{I\Sigma}_n$ proves the $\Sigma_n$-collection schema for $n>0$. $\mathit{I\Sigma}^R_{n+1}$ contains $\mathit{I\Sigma}_n$, hence the equivalence holds in $T+\sir{n+1}$ as well.

Further, it is easy to see from our assumptions that $T$ proves $\ex{v}\phi'(0,v)$ and $\al{z}(\ex{v}\phi'(z,v)\to \ex{v}\phi'(s(z),v))$. Hence, $T+\sir{n+1}$ proves $\ex{v}\phi''(0,v)$ and $\al{z}(\ex{v}\phi''(z,v)\to \ex{v}\phi''(s(z),v))$. By $\sir{n+1}$, we conclude $\al{z}\ex{v}\phi''(z,v)$, from which $\al{z}\ex{v}\phi'(z,v)$ and $\al{z}\psi(z)$ follow. 
\end{proof}
Applying Lemma~\ref{b_induction} to $\zeta$ we obtain $\al{z}\zeta(z)$, hence $\theta$ and the claim of the theorem. 
\end{proof}

Notice that, by Theorem \ref{main3}, the system $\ST_0$ is deductively equivalent to the fragment of arithmetic $\mathit{I\Delta}_0$. 

\section{Conclusion} 
In this paper we have defined a number of cyclic systems axiomatizing the main fragments of Peano arithmetic with the axioms or rules of induction restricted to classes of the arithmetical hierarchy. This opens a number of directions for further research. 

Firstly, we would like to do some experiments with inductive proof search based on the cyclic systems of arithmetic and type theory akin to those introduced in this paper. Some preliminary work in this direction has already been done (see~\cite{Smir23}). 

Secondly, we would like to study the possibilities of using cyclic proofs for proof-theoretic analysis of arithmetic and other systems. In particular, there is a need to develop (partial) cut-elimination techniques in this context.  

\hbadness 10000
\emergencystretch=1.5em
\printbibliography

\end{document}